\documentclass[12pt]{article}

\usepackage[T1]{fontenc}
\usepackage{lmodern}
\usepackage[margin=1in]{geometry}
\usepackage{microtype}
\usepackage{amsmath,amssymb,amsthm,mathtools}
\usepackage{enumitem}
\usepackage{array}
\usepackage{booktabs}
\usepackage{aliascnt}

\newtheorem{theorem}{Theorem}[section]
\newaliascnt{lemma}{theorem}
\newtheorem{lemma}[lemma]{Lemma}
\aliascntresetthe{lemma}
\newaliascnt{proposition}{theorem}
\newtheorem{proposition}[proposition]{Proposition}
\aliascntresetthe{proposition}
\newaliascnt{corollary}{theorem}
\newtheorem{corollary}[corollary]{Corollary}
\aliascntresetthe{corollary}
\theoremstyle{definition}
\newaliascnt{remark}{theorem}
\newtheorem{remark}[remark]{Remark}
\aliascntresetthe{remark}

\usepackage[capitalise]{cleveref}

\newcommand{\Z}{\mathbb Z}
\newcommand{\R}{\mathbb R}
\newcommand{\T}{\mathbb T}
\newcommand{\one}{\mathbf 1}
\newcommand{\wh}{\widehat}
\newcommand{\ee}{\mathrm e}
\newcommand{\card}[1]{\lvert #1\rvert}

\title{Both directions of Fuglede's conjecture fail in dimension two}
\author{Tao Zhang\thanks{Email: \texttt{zhant220@163.com}.}\\
\small Institute of Mathematics and Interdisciplinary Sciences,\\[-1mm]
\small Xidian University, Xi'an 710126, China}
\date{}

\begin{document}

\maketitle

\begin{abstract}
Fuglede's conjecture asserts that a measurable set of positive and finite measure is spectral if and only if it tiles Euclidean space by translations. Counterexamples are known in every dimension $d\ge3$, whereas the one- and two-dimensional cases have remained unresolved. We construct two explicit $60$-point subsets of the rank-two finite Abelian group $\Z_{60}\times\Z_{12}$: one is a translational tile with no spectrum, and the other is spectral but does not tile. A finite-to-infinite transference principle lifts them to bounded subsets of $\R^2$ that are finite unions of unit squares. Consequently, both implications in Fuglede's conjecture fail in dimension two.


\end{abstract}

\section{Introduction}

Let $\Omega\subset\R^d$ be a bounded measurable set with $0<\lvert\Omega\rvert<\infty$. The set $\Omega$ is called \emph{spectral} if there is a set $\Lambda\subset\R^d$ such that
\[
  \left\{\lvert\Omega\rvert^{-1/2}
  \ee^{2\pi i\langle\lambda,x\rangle}:\lambda\in\Lambda\right\}
\]
is an orthonormal basis of $L^2(\Omega)$. In this case, $\Lambda$ is called a \emph{spectrum} of $\Omega$. The set $\Omega$ is a \emph{translational tile} if there is a discrete set $\mathcal T\subset\R^d$ for which
\[
  \sum_{t\in\mathcal T}\one_\Omega(x-t)=1
\]
for almost every $x\in\R^d$.

Fuglede conjectured in 1974 that these two properties are equivalent \cite{F74}. His theorem proves the equivalence when either the spectrum or the translation set is a lattice, and the conjecture has since linked harmonic analysis with translational tiling, additive combinatorics, and factorization theory. A substantial positive theory survives under geometric or arithmetic hypotheses. In one dimension, the problem is closely tied to factorizations of subsets of the integers and cyclic groups \cite{CM99,L02,LW1997,PW2001}. The conjecture holds for convex planar domains \cite{IKT03}, for convex polytopes in dimension three \cite{GL17}, and, more generally, for convex domains in every dimension \cite{LM21}. Further structural reductions are discussed in \cite{DL2014,K2025,K2026}.

For unrestricted measurable sets, however, the higher-dimensional picture is negative. Tao constructed a spectral set that does not tile in $\R^d$ for $d\ge5$ \cite{T04}. The dimension was subsequently lowered to four by Matolcsi \cite{M05} and to three by Kolountzakis and Matolcsi \cite{KM2006}. In the reverse direction, non-spectral tiles were obtained in dimensions at least five by Kolountzakis and Matolcsi \cite{KM06}, in dimension four by Farkas and R\'ev\'esz \cite{FR06}, and in dimension three by Farkas, Matolcsi, and M\'ora \cite{FMM06}. Thus both implications were known to fail in every dimension $d\ge3$, while the one- and two-dimensional cases remained open.

The known counterexamples are built first in finite Abelian groups and then transferred to lattices and Euclidean space. This strategy makes the number of cyclic factors decisive: a counterexample in
\[
  \Z_{n_1}\times\cdots\times\Z_{n_d}
\]
produces a counterexample in $\Z^d$ and hence in $\R^d$. To reach the plane, one therefore needs a counterexample in a group with only two cyclic factors. This is a restrictive setting. Fuglede's conjecture is known to hold in many finite Abelian groups of small rank, including several finite-field, rank-two, and cyclic families; representative results include \cite{AABF17,FKS2012,IMP17,KMSV20,KMSV2012,KS2021,L02,M21,MK17,S19,S20,Somlai21,Z2022,Z2024,ZhouP2}. Related positive results are also known in non-Euclidean locally compact Abelian groups, for example in $\mathbb{Q}_p$ \cite{FFLS19}.
Recent results further sharpen this positive picture, the universal-spectrum property passes from $G$ to $G\times\Z_p$ when $p\nmid\card{G}$ \cite{ZhouUniversal}. The constructions below cross this rank-two barrier.

Our main result is the following.

\begin{theorem}\label{main-thm}
There exists a bounded translational tile in $\R^2$ that is not spectral, and there exists a bounded spectral set in $\R^2$ that does not tile by translations.
\end{theorem}

The two constructions arise from the same $12$-point graph $E\subset\Z_{12}^2$. On the tiling side, $E$ has three tiling complements with no common spectrum. Placing these complements in five layers produces a $60$-point tile that cannot be spectral. On the spectral side, the same set $E$ has three spectra with no common tiling complement; a parallel five-layer construction produces a $60$-point spectral set that cannot tile. The auxiliary factor $\Z_5$ does not increase the rank, because the Chinese remainder theorem gives
\[
  \Z_{12}^2\times\Z_5\cong\Z_{60}\times\Z_{12}.
\]
The transference result in \Cref{prop:transference} then carries both finite examples to the plane. 

The rest of the paper follows this scheme. \Cref{sec-pre} records the finite Abelian group criteria, proves the transference statement in the required rectangular form, and isolates two layer-amplification principles. \Cref{sec:tile-not-spectral} constructs the non-spectral tile, and \Cref{sec:spectral-not-tile} constructs the spectral set that does not tile. 

\section{Preliminaries}\label{sec-pre}

Let
\[
  G=\Z_{n_1}\times\cdots\times\Z_{n_d}.
\]
We identify the dual group $\wh{G}$ with $G$ through the characters
\[
  \chi_\xi(x)=\exp\left(2\pi i\sum_{j=1}^d\frac{\xi_jx_j}{n_j}\right).
\]
For $A\subseteq G$, define
\[
  \wh{\one_A}(\xi)=\sum_{a\in A}\chi_\xi(a),
  \qquad
  \mathcal Z_A=\{\xi\in G:\wh{\one_A}(\xi)=0\}.
\]
A set $\Lambda\subseteq\wh{G}$ is a spectrum of $A$ if the restrictions to $A$ of the characters indexed by $\Lambda$ form an orthogonal basis of $\mathbb C^A$. Equivalently,
\[
  \card{\Lambda}=\card{A},
  \qquad
  (\Lambda-\Lambda)\setminus\{0\}\subseteq\mathcal Z_A.
\]
Since the corresponding character matrix is square, transposition gives the useful symmetry
\begin{equation}\label{eq:spectral-symmetry}
  \Lambda\text{ is a spectrum of }A
  \quad\Longleftrightarrow\quad
  A\text{ is a spectrum of }\Lambda.
\end{equation}
A set $A$ tiles $G$ if there is a set $B\subseteq G$ such that every $g\in G$ has a unique representation $g=a+b$, with $a\in A$ and $b\in B$. We write
\[
  A+B=G.
\]
Equivalently,
\begin{equation}\label{eq:tile-difference}
  \card{A}\card{B}=\card{G},
  \qquad
  (A-A)\cap(B-B)=\{0\}.
\end{equation}

If $T\subseteq G$ is a tile, a set $S\subseteq\wh{G}$ is a \emph{universal spectrum} for $T$ if it is a spectrum of every tiling complement of $T$. Dually, if $E$ is spectral, a set $C\subseteq G$ is a \emph{universal tiling complement} for $E$ if it tiles with every spectrum of $E$. These notions and the amplification of their failure are closely related to the universal spectrum method of \cite{FMM06,KM06,PW2001}.

For $A\subseteq G=\Z_{n_1}\times\Z_{n_2}$, choose representatives of its elements in
$[0,n_1)\times[0,n_2)$. For $k\ge1$, set
\begin{equation}\label{eq:block-lift}
  Q_k=\{(n_1j_1,n_2j_2):0\le j_1,j_2<k\},
  \qquad
  A[k]=A+Q_k\subset\Z^2.
\end{equation}
A finite set $F\subset\Z^2$ is called spectral in $\Z^2$ if it has a spectrum in $\wh{\Z^2}=\T^2$.

\begin{proposition}[Finite-to-infinite transference]\label{prop:transference}
Let $\varnothing\ne A\subseteq G=\Z_{n_1}\times\Z_{n_2}$ and define $A[k]$ by \eqref{eq:block-lift}.
\begin{enumerate}[label=\textnormal{(\roman*)}]
  \item There exists $k_{\mathrm{sp}}$ such that, for every $k\ge k_{\mathrm{sp}}$, the set $A[k]$ is spectral in $\Z^2$ if and only if $A$ is spectral in $G$.
  \item If $A$ tiles $G$, then $A[k]$ tiles $\Z^2$ for every $k\ge1$. If $A$ does not tile $G$, then $A[k]$ does not tile $\Z^2$ for every integer
  \[
    k>\max\left\{2,\frac{8n_1n_2}{\card{A}}\right\}.
  \]
  \item For every nonempty finite $F\subset\Z^2$, we have
  \[
    F\text{ is spectral in }\Z^2
    \quad\Longleftrightarrow\quad
    F+[0,1)^2\text{ is spectral in }\R^2,
  \]
  and
  \[
    F\text{ tiles }\Z^2
    \quad\Longleftrightarrow\quad
    F+[0,1)^2\text{ tiles }\R^2.
  \]
\end{enumerate}
\end{proposition}

\begin{proof}
Part~(i) is the two-dimensional case of \cite[Theorem~4.1]{KM06}. We record the explicit forward construction, which is valid for every $k\ge1$ and will be used later. If $\Lambda\subseteq\wh{G}$ is a spectrum of $A$, then
\begin{equation}\label{eq:explicit-spectrum-lift}
  \Lambda[k]
  =\left\{
  \left(\frac{\lambda_1+s_1/k}{n_1},
        \frac{\lambda_2+s_2/k}{n_2}\right):
  \lambda\in\Lambda,\ 0\le s_1,s_2<k
  \right\}\subset\T^2
\end{equation}
is a spectrum of $A[k]$. Indeed, the Fourier sum over $A[k]=A+Q_k$ factors into the Fourier sum over $A$ and two geometric sums. If two frequencies have different $s$-indices, one of the geometric sums vanishes; if their $s$-indices agree, orthogonality reduces to that of $\Lambda$ on $A$. The converse is the nontrivial part of \cite[Theorem~4.1]{KM06}: for all sufficiently large $k$, a spectrum of $A[k]$ yields, after the frequency-cell decomposition used there and projection to $\wh G$, a spectrum of $A$.

For the first assertion in part~(ii), suppose that $A+B=G$. Choosing representatives of $B$ in $[0,n_1)\times[0,n_2)$ gives the explicit periodic direct tiling
\[
  A[k]+\bigl(B+(kn_1\Z)\times(kn_2\Z)\bigr)=\Z^2.
\]

We prove the converse in a quantitative form. This is the rectangular version of \cite[Proposition~2.5]{M05}. Put
\[
  L=n_1\Z\times n_2\Z,
  \qquad
  P=([0,n_1)\times[0,n_2))\cap\Z^2,
  \qquad
  R_k=([0,kn_1)\times[0,kn_2))\cap\Z^2.
\]
Then $P$ is a complete set of representatives for $\Z^2/L$, and
\[
  A[k]=(A+L)\cap R_k,
  \qquad
  \card{A[k]}=\card{A}\,k^2.
\]
Fix an integer $k\ge3$ for which
\begin{equation}\label{eq:alpha-k}
  \alpha_k:=\frac{8n_1n_2}{\card{A}\,k}<1,
\end{equation}
and suppose, toward a contradiction, that $A[k]$ tiles $\Z^2$ with translation set $\Sigma$.

For an integer $N>\max\{kn_1,kn_2\}$, let
\[
  C_N=[0,N)^2\cap\Z^2,
  \qquad
  \Sigma_N=\{\sigma\in\Sigma:(\sigma+A[k])\cap C_N\ne\varnothing\}.
\]
The translates indexed by $\Sigma_N$ are pairwise disjoint. If one of them meets $C_N$, it is contained in
\[
  [-kn_1,N+kn_1)\times[-kn_2,N+kn_2),
\]
so counting lattice points gives
\begin{equation}\label{eq:sigma-count}
  \card{\Sigma_N}
  \le
  \frac{(N+2kn_1)(N+2kn_2)}{\card{A}\,k^2}.
\end{equation}
Define the boundary ring
\[
  \mathcal R_k=
  \bigl([-n_1,(k+1)n_1)\times[-n_2,(k+1)n_2)\bigr)
  \setminus
  \bigl([n_1,(k-1)n_1)\times[n_2,(k-1)n_2)\bigr),
\]
where both rectangles are intersected with $\Z^2$. Since the outer and inner rectangles have side-length factors $k+2$ and $k-2$, respectively,
\begin{equation}\label{eq:ring-count}
  \card{\mathcal R_k}
  =\bigl((k+2)^2-(k-2)^2\bigr)n_1n_2
  =8kn_1n_2.
\end{equation}
Let
\[
  D_N=([0,N-n_1)\times[0,N-n_2))\cap\Z^2.
\]
By \eqref{eq:sigma-count}--\eqref{eq:ring-count},
\begin{align*}
  \card{\bigcup_{\sigma\in\Sigma_N}(\sigma+\mathcal R_k)}
  \le \card{\Sigma_N}\card{\mathcal R_k}\le \alpha_k (N+2kn_1)(N+2kn_2).
\end{align*}
Because $\alpha_k<1$, one may choose $N$ so large that
\begin{equation}\label{eq:uncovered-corner-ineq}
  \alpha_k (N+2kn_1)(N+2kn_2)
  <(N-n_1)(N-n_2)=\card{D_N}.
\end{equation}
Choose
\[
  x\in D_N\setminus\bigcup_{\sigma\in\Sigma_N}(\sigma+\mathcal R_k)
\]
and set $P_x=x+P$. Then $P_x\subset C_N$.

Let
\[
  I_x=\{\sigma\in\Sigma:(\sigma+A[k])\cap P_x\ne\varnothing\}.
\]
Since $P_x\subset C_N$, every $\sigma\in I_x$ belongs to $\Sigma_N$. Moreover, $A[k]\subset R_k$, so $(\sigma+R_k)\cap P_x\ne\varnothing$. Coordinatewise overlap of these two rectangles places $x$ in the outer rectangle defining $\sigma+\mathcal R_k$. Since $x\notin\sigma+\mathcal R_k$, it lies in the corresponding inner rectangle; hence
\begin{equation}\label{eq:cell-inside-window}
  P_x\subseteq\sigma+R_k
  \qquad(\sigma\in I_x).
\end{equation}
It follows from $A[k]=(A+L)\cap R_k$ that, on $P_x$,
\begin{equation}\label{eq:truncated-periodic-agreement}
  (\sigma+A[k])\cap P_x=(\sigma+A+L)\cap P_x
  \qquad(\sigma\in I_x).
\end{equation}

Define
\[
  B_x=\{\sigma+L:\sigma\in I_x\}\subset\Z^2/L.
\]
First, distinct elements of $I_x$ have distinct residues modulo $L$. Indeed, if $\sigma-\sigma'\in L$, then $\sigma+A+L=\sigma'+A+L$. Choosing a point of $(\sigma+A[k])\cap P_x$ and using \eqref{eq:cell-inside-window}--\eqref{eq:truncated-periodic-agreement} for both translations would place that point in two distinct translates of the tiling, a contradiction.

Now let $g\in\Z^2/L$, and let $z_g\in P_x$ be its unique representative. The tiling of $\Z^2$ supplies a unique $\sigma\in I_x$ with $z_g\in\sigma+A[k]$. By \eqref{eq:truncated-periodic-agreement}, this gives a representation of $g$ in $A+B_x$. Conversely, two representations of $g$ in $A+B_x$ would, again by \eqref{eq:cell-inside-window}--\eqref{eq:truncated-periodic-agreement}, place $z_g$ in two translates $\sigma+A[k]$; uniqueness of the original tiling then forces the two translations and the two elements of $A$ to coincide. Thus
\[
  A+B_x=\Z^2/L\cong G
\]
is a direct factorization. Therefore, if $A$ does not tile $G$, no such tiling of $A[k]$ can exist whenever \eqref{eq:alpha-k} holds. This proves part~(ii).

The spectral equivalence in part~(iii) is \cite[Theorem~4.2]{KM06}. In particular, if $\Lambda\subset\T^2$ is a spectrum of $F$, then $\Lambda+\Z^2$ is a spectrum of $F+[0,1)^2$.

It remains to prove the tiling equivalence. The forward implication is immediate: if $F+B=\Z^2$, then
\[
  (F+[0,1)^2)+B=\R^2
\]
is a tiling. Conversely, suppose that $\Omega=F+[0,1)^2$ tiles $\R^2$ with translation set $\mathcal T$. The set $\mathcal T$ is locally finite. Indeed, fix $f_0\in F$. If distinct $t,t'\in\mathcal T$ satisfied $\lVert t-t'\rVert_\infty<1$, then the open unit squares
\[
  t+f_0+(0,1)^2
  \quad\text{and}\quad
  t'+f_0+(0,1)^2
\]
would overlap in positive measure, contradicting the almost-everywhere disjointness of the tiling. Hence $\mathcal T$ is uniformly separated, and therefore countable and locally finite.

Let $N_0$ be a null set outside which the tiling identity holds, and put
\[
  \mathcal B=N_0\cup\bigcup_{t\in\mathcal T}\partial(t+\Omega).
\]
This is a null set. Choose
\begin{equation}\label{eq:generic-x}
  x\in[0,1)^2\setminus
  \bigcup_{m\in\Z^2}(\mathcal B-m).
\end{equation}
Then every point of $x+\Z^2$ avoids both the exceptional set and all translate boundaries. Write each $t\in\mathcal T$ uniquely as
\[
  t=m_t+\alpha_t,
  \qquad
  m_t\in\Z^2,
  \quad
  \alpha_t\in[0,1)^2,
\]
and define $\varepsilon_x(t)\in\{0,1\}^2$ by
\[
  (\varepsilon_x(t))_j=
  \begin{cases}
    1,&x_j<(\alpha_t)_j,\\
    0,&x_j>(\alpha_t)_j.
  \end{cases}
\]
The equality case is excluded by \eqref{eq:generic-x}. A coordinatewise check gives
\begin{equation}\label{eq:lattice-slice}
  (t+\Omega)\cap(x+\Z^2)
  =x+\bigl(F+m_t+\varepsilon_x(t)\bigr).
\end{equation}
Consequently, for every $n\in\Z^2$, the point $x+n$ belongs to exactly one translate $t+\Omega$, and \eqref{eq:lattice-slice} says that $n$ belongs to exactly one of the sets
\[
  F+b_x(t),
  \qquad
  b_x(t):=m_t+\varepsilon_x(t).
\]
The map $t\mapsto b_x(t)$ is injective: if two distinct translations had the same image, then for any $f\in F$ the point $x+f+b_x(t)$ would belong to both translates, contrary to the choice of $x$. Hence
\[
  F+\{b_x(t):t\in\mathcal T\}=\Z^2
\]
is a direct translational tiling. This completes the proof.
\end{proof}

\begin{remark}\label{rem:transference-consequence}
The spectral statements in \Cref{prop:transference} are \cite[Theorems~4.1 and~4.2]{KM06}. The finite-to-lattice non-tiling argument is the rectangular analogue of \cite[Proposition~2.5]{M05}, with the explicit sufficient bound $k>8n_1n_2/\card A$. The two-way transference is also summarized in \cite[Theorems~1--2 and Remark~1]{FMM06}.
\end{remark}

\subsection{Layer-amplification principles}\label{subsec:layer-amplification}

Let $H$ be a finite Abelian group and let $q\ge2$. For $X\subseteq H\times\Z_q$, write
\[
  X_t=\{h\in H:(h,t)\in X\}
  \qquad(t\in\Z_q)
\]
for its $H$-layers.

\begin{lemma}\label{lem:layer-common-spectrum}
Let $B_0,\ldots,B_{r-1}\subseteq H$ be $m$-point sets, let $\jmath:\Z_q\to\{0,\ldots,r-1\}$ be surjective, and define
\[
  \mathcal A=\bigcup_{t\in\Z_q}(B_{\jmath(t)}\times\{t\}).
\]
Assume that a set $E\subseteq H$ satisfies $E+B_j=H$ for every $j$. Then
\[
  (E\times\{0\})+\mathcal A=H\times\Z_q.
\]
If, in addition,
\begin{equation}\label{eq:layer-zero-implication}
  \wh{\one_{\mathcal A}}(\xi,0)=0
  \quad\Longrightarrow\quad
  \wh{\one_{B_j}}(\xi)=0\ \text{ for every }j,
\end{equation}
and the sets $B_0,\ldots,B_{r-1}$ have no common $m$-point spectrum, then $\mathcal A$ is not spectral.
\end{lemma}

\begin{proof}
The tiling assertion follows layer by layer from the direct factorizations $E+B_j=H$. Suppose that $\Lambda\subseteq\wh H\times\wh{\Z_q}$ were a spectrum of $\mathcal A$. Since $\card{\mathcal A}=qm$, one fiber of $\Lambda$ over $\wh{\Z_q}$ contains at least $m$ points. Select $m$ of them. Their projections to $\wh H$ are distinct, and the difference of any two selected frequencies has last coordinate zero. Spectral orthogonality and \eqref{eq:layer-zero-implication} therefore place every nonzero projected difference in the zero set of every $B_j$. The $m$ projections form a common spectrum of the $B_j$, a contradiction.
\end{proof}

\begin{lemma}\label{lem:layer-common-complement}
Let $V_0,\ldots,V_{r-1}\subseteq H$ be $m$-point sets, let $\jmath:\Z_q\to\{0,\ldots,r-1\}$ be surjective, and define
\[
  \mathcal A=\bigcup_{t\in\Z_q}(V_{\jmath(t)}\times\{t\}).
\]
If an $m$-point set $E\subseteq\wh H$ is a spectrum of every $V_j$, then $E\times\wh{\Z_q}$
is a spectrum of $\mathcal A$. If, moreover,
\begin{equation}\label{eq:layer-full-slab}
  H\times(\Z_q\setminus\{0\})\subseteq\mathcal A-\mathcal A
\end{equation}
and the sets $V_0,\ldots,V_{r-1}$ have no common tiling complement in $H$, then $\mathcal A$ does not tile $H\times\Z_q$.
\end{lemma}

\begin{proof}
The two sets $\mathcal A$ and $E\times\wh{\Z_q}$ both have $qm$ elements. Consider two distinct frequencies $(e,\rho)$ and $(e',\rho')$. If $e\ne e'$, then the Fourier sum over each layer $V_{\jmath(t)}$ vanishes because $E$ is a spectrum of that layer. If $e=e'$ and $\rho\ne\rho'$, each layer sum equals $m$, and summation over $t\in\Z_q$ is a complete nontrivial character sum. Thus all distinct frequencies are orthogonal, proving spectrality.

Suppose that $\mathcal A+T=H\times\Z_q$. By \eqref{eq:tile-difference} and \eqref{eq:layer-full-slab}, no two elements of $T$ can have different last coordinates. After translating $T$, write $T=C\times\{0\}$. Comparing the $q$ layers of the tiling gives
\[
  V_{\jmath(t)}+C=H
  \qquad(t\in\Z_q).
\]
Since $\jmath$ is surjective, $C$ is a common tiling complement of all the $V_j$, a contradiction.
\end{proof}

\section{A tile that is not spectral}\label{sec:tile-not-spectral}

Both finite counterexamples will live in
\[
  \mathcal G=\Z_{60}\times\Z_{12},
  \qquad \card{\mathcal G}=720.
\]
We first construct the tiling example.

\begin{theorem}\label{thm:main-tile}
There exist explicit sets $A_-,T_-\subseteq\mathcal G$ with
\[
  \card{A_-}=60,
  \qquad
  \card{T_-}=12,
  \qquad
  A_-+T_-=\mathcal G,
\]
such that $A_-$ is not spectral.
\end{theorem}

\subsection{A seed tile with incompatible complements}

Let
\[
  H=\Z_{12}\times\Z_{12}.
\]
Define $f:\Z_{12}\to\Z_{12}$ by
\[
  \bigl(f(0),f(1),\ldots,f(11)\bigr)
  =(0,5,0,7,9,2,10,4,7,8,10,10),
\]
and put
\[
  E=\{(f(y),y):y\in\Z_{12}\}.
\]
Consider the following three $12$-point sets:
\begin{align}
  B_0&=\{(x,0):x\in\Z_{12}\},\notag\\
  B_1&=\{(x,-x):x\in\Z_{12}\},\notag\\
  B_2&=\{(3c+4a,\,4a+6b):a\in\Z_3,\ b,c\in\Z_2\}.\label{eq:B2}
\end{align}

\begin{lemma}\label{lem:three-tilings}
For $j=0,1,2$, one has $E+B_j=H$.
\end{lemma}

\begin{proof}
The assertion for $B_0$ follows immediately because $E$ is a graph over its second coordinate.

For $B_1$, observe that
\begin{equation}\label{eq:f-plus-y}
  \bigl(f(y)+y\bigr)_{y=0}^{11}
  =(0,6,2,10,1,7,4,11,3,5,8,9),
\end{equation}
which is a permutation of $\Z_{12}$. Given $(u,v)\in H$, the identity
\[
  (u,v)=(f(y),y)+(x,-x)
\]
forces $u+v=f(y)+y$, which determines $y$ uniquely; then $x=u-f(y)$ is also uniquely determined.

For $B_2$, let
\[
  P=\{(4a,4a+6b):a\in\Z_3,\ b\in\Z_2\},
\]
so that $B_2=P+\{(0,0),(3,0)\}$. Consider the homomorphism
\[
  \phi:H\longrightarrow\Z_2\times\Z_{12},
  \qquad
  \phi(x,y)=(x-y\bmod2,\ x+2y\bmod12).
\]
A direct computation gives $\ker \phi=P$ and 
\begin{align*}
  \phi(E)=\{&(0,0),(0,1),(0,4),(0,6),(0,7),(0,10),\\
         &(1,0),(1,2),(1,5),(1,6),(1,8),(1,11)\}.
\end{align*}
Since $\phi(3,0)=(1,3)$, the two sets $\phi(E)$ and $\phi(E)+(1,3)$ are disjoint and together exhaust $\Z_2\times\Z_{12}$. Lifting this direct factorization through $\phi$ gives $E+B_2=H$.
\end{proof}

Let $\zeta=\ee^{2\pi i/12}$. For $(u,v)\in H$, the Fourier transforms of the three complements are
\begin{align}
  \wh{\one_{B_0}}(u,v)&=12\,\one_{\{u=0\}},\label{eq:B0hat}\\
  \wh{\one_{B_1}}(u,v)&=12\,\one_{\{u=v\}},\label{eq:B1hat}\\
  \wh{\one_{B_2}}(u,v)&=(1+\zeta^{3u})(1+\zeta^{6v})
  (1+\zeta^{4(u+v)}+\zeta^{8(u+v)}).\label{eq:B2hat}
\end{align}
The last identity follows by separating the sums over $c$, $b$, and $a$ in \eqref{eq:B2}.

\begin{proposition}\label{prop:no-common-spectrum}
There is no $12$-point set $\Lambda\subseteq\wh{H}$ that is simultaneously a spectrum of $B_0$, $B_1$, and $B_2$.
\end{proposition}

\begin{proof}
Assume that such a set $\Lambda$ exists, and write its elements as $(u,v)$. By \eqref{eq:B0hat}, two distinct elements of $\Lambda$ cannot have the same first coordinate. Hence the twelve $u$-coordinates are exactly the elements of $\Z_{12}$. Similarly, \eqref{eq:B1hat} implies that the twelve values $u-v$ also run through all of $\Z_{12}$.

Partition $\Lambda$ into six classes according to
\[
  \beta=u+v\pmod3,
  \qquad
  \gamma=v\pmod2.
\]
If $(u,v)$ and $(u',v')$ lie in the same class, then
\[
  (u-u')+(v-v')\equiv0\pmod3,
  \qquad
  v-v'\equiv0\pmod2.
\]
Thus the last two factors in \eqref{eq:B2hat}, evaluated at their difference, are equal to $2$ and $3$, respectively, and are therefore nonzero. Orthogonality on $B_2$ forces
\[
  1+\zeta^{3(u-u')}=0,
\]
or equivalently
\[
  u-u'\equiv2\pmod4.
\]
No set of three residues modulo $4$ has all pairwise differences congruent to $2$ modulo $4$. Therefore, each of the six classes has at most two elements. Since $\card{\Lambda}=12$, every class has exactly two elements, and the two $u$-coordinates in each class have the same parity.

Let $a$ be the number of classes with $\gamma=0$ in which the common parity of $u$ is odd, and define $b$ analogously for $\gamma=1$. Since the $u$-coordinates run through $\Z_{12}$, exactly six of them are odd. Therefore
\[
  2(a+b)=6,
  \qquad a+b=3.
\]
The values $u-v$ also run through $\Z_{12}$, so exactly six of them are odd. In a class with $\gamma=0$, the parity of $u-v$ is the parity of $u$; in a class with $\gamma=1$, it is the opposite parity. Hence
\[
  2\bigl(a+(3-b)\bigr)=6,
  \qquad a=b.
\]
The relations $a+b=3$ and $a=b$ are incompatible for integers $a,b$. This contradiction proves the proposition.
\end{proof}

Thus the tile $E$ has three tiling complements with no common spectrum. We now amplify this incompatibility into a tile that is itself non-spectral.

\subsection{Five-layer amplification and passage to the plane}

Let
\[
  K=H\times\Z_5
\]
and define
\[
  \widetilde A_-
  =(B_0\times\{0\})
  \cup(B_1\times\{3,4\})
  \cup(B_2\times\{1,2\}).
\]

\begin{proposition}\label{prop:Atilde-minus}
The set $\widetilde A_-$ tiles $K$ but is not spectral.
\end{proposition}

\begin{proof}
By \Cref{lem:three-tilings}, every $H$-layer is tiled by $E$. Thus the tiling assertion is also the first conclusion of \Cref{lem:layer-common-spectrum}, with the layer assignment
\[
  \jmath(0)=0,
  \qquad
  \jmath(1)=\jmath(2)=2,
  \qquad
  \jmath(3)=\jmath(4)=1,
\]
and in particular
\begin{equation}\label{eq:Atilde-minus-tiling}
  (E\times\{0\})+\widetilde A_-=K.
\end{equation}

At a frequency with zero $\Z_5$-coordinate,
\[
  \wh{\one_{\widetilde A_-}}(u,v,0)
  =\wh{\one_{B_0}}(u,v)
  +2\wh{\one_{B_1}}(u,v)
  +2\wh{\one_{B_2}}(u,v).
\]
We claim that
\begin{equation}\label{eq:weighted-zero-equivalence}
  \wh{\one_{\widetilde A_-}}(u,v,0)=0
  \quad\Longleftrightarrow\quad
  \wh{\one_{B_0}}(u,v)
  =\wh{\one_{B_1}}(u,v)
  =\wh{\one_{B_2}}(u,v)=0.
\end{equation}
If $u\ne0$ and $u\ne v$, the $B_0$- and $B_1$-terms vanish, so the equivalence is immediate. If $u=0$, the $B_0$-term is $12$. From \eqref{eq:B2hat}, the $B_2$-term is nonzero only for $v=0$ or $v=6$, and in both cases it is the positive real number $12$; the $B_1$-term is nonzero only for $v=0$. Hence cancellation is impossible.

It remains to treat $u=v\ne0$. In this case $\wh{\one_{B_2}}(u,u)=0$, as follows directly from the three factors in \eqref{eq:B2hat}:
\begin{itemize}[leftmargin=2em]
  \item if $u$ is odd, then $1+\zeta^{6u}=0$;
  \item if $u\equiv2\pmod4$, then $1+\zeta^{3u}=0$;
  \item the remaining nonzero residues are $u=4,8$, and then
  $1+\zeta^{4u}+\zeta^{8u}=1+\zeta^4+\zeta^8=0$.
\end{itemize}
The weighted $B_1$-term is therefore $24$, so the left-hand side is nonzero, exactly as required. This proves \eqref{eq:weighted-zero-equivalence}.

Now \eqref{eq:weighted-zero-equivalence} gives the implication \eqref{eq:layer-zero-implication}, while \Cref{prop:no-common-spectrum} rules out a common $12$-point spectrum. The second conclusion of \Cref{lem:layer-common-spectrum} shows that $\widetilde A_-$ is not spectral.
\end{proof}

Consider the Chinese remainder isomorphism
\begin{equation}\label{eq:crt-physical}
  \Phi:H\times\Z_5\longrightarrow\mathcal G,
  \qquad
  \Phi(x,y,t)=(25x+36t\bmod60,\ y\bmod12).
\end{equation}
Indeed, $25x+36t$ is congruent to $x$ modulo $12$ and to $t$ modulo $5$. Define
\[
  A_-=\Phi(\widetilde A_-),
  \qquad
  T_-=\Phi(E\times\{0\}).
\]
More explicitly,
\begin{align}
  A_-={}&\{(25x,0):x\in\Z_{12}\}\notag\\
  &\cup\{(25x+36t,-x):t\in\{3,4\},\ x\in\Z_{12}\}\notag\\
  &\cup\{(25(3c+4a)+36t,\,4a+6b):
  t\in\{1,2\},\ a\in\Z_3,\ b,c\in\Z_2\},\label{eq:A-minus-explicit}
\end{align}
and
\begin{equation}\label{eq:T-minus-explicit}
  T_-=\{(25f(y),y):y\in\Z_{12}\}.
\end{equation}
All first coordinates in \eqref{eq:A-minus-explicit}--\eqref{eq:T-minus-explicit} are understood modulo $60$, and all second coordinates are understood modulo $12$.

\begin{proof}[Proof of \Cref{thm:main-tile}]
The map $\Phi$ is a group isomorphism. Therefore \eqref{eq:Atilde-minus-tiling} gives
\[
  A_-+T_-=\mathcal G,
\]
and \Cref{prop:Atilde-minus} shows that $A_-$ is not spectral.
\end{proof}

\begin{corollary}\label{cor:plane-tile}
There exists a bounded translational tile in $\R^2$ that is not spectral. Moreover, it may be chosen as a finite union of unit squares.
\end{corollary}

\begin{proof}
Apply \Cref{prop:transference} to $A_-\subset\Z_{60}\times\Z_{12}$. Since $A_-$ tiles the finite group, $A_-[k]$ tiles $\Z^2$ for every $k\ge1$. Since $A_-$ is not spectral, the converse spectral transfer gives a threshold $k_{\mathrm{sp}}$ such that, for every $k\ge k_{\mathrm{sp}}$,
\[
  A_-[k]
  =A_-+\{(60j,12\ell):0\le j,\ell<k\}
\]
is not spectral. Thus, for every such $k$,
\[
  \Omega_-^{(k)}=A_-[k]+[0,1)^2
\]
is a non-spectral translational tile of $\R^2$. 
\end{proof}

\section{A spectral set that does not tile}\label{sec:spectral-not-tile}

We retain the notation $H$, $E$, $K$, $\mathcal G$, and $\Phi$ from the preceding section.

\begin{theorem}\label{thm:main-spectral}
There exist explicit sets $A_+,\Lambda_+\subseteq\mathcal G$ with
\[
  \card{A_+}=\card{\Lambda_+}=60,
\]
such that $\Lambda_+$ is a spectrum of $A_+$, but $A_+$ does not tile $\mathcal G$.
\end{theorem}

\subsection{A seed spectral configuration with incompatible tiling complements}

Define
\begin{align*}
  V_0&=\{(0,k):k\in\Z_{12}\},\\
  V_1&=\{(k,k):k\in\Z_{12}\},\\
  V_2&=H_0\cup(c+H_0),
\end{align*}
where $H_0=\langle(4,2)\rangle$ and $c=(1,2)$.
Thus $\card{V_0}=\card{V_1}=\card{V_2}=12$.

\begin{lemma}\label{lem:three-spectra}
Each of $V_0$, $V_1$, and $V_2$ is a spectrum of $E$. Equivalently, $E$ is a spectrum of each of $V_0$, $V_1$, and $V_2$.
\end{lemma}

\begin{proof}
Since the second coordinates of $E$ run through $\Z_{12}$, the character matrix indexed by $V_0$ is the Fourier matrix of order $12$. Thus $V_0$ is a spectrum of $E$.

For $V_1$, the relevant coordinate on a point $(f(y),y)$ is $f(y)+y$. By \eqref{eq:f-plus-y}, these values also run through $\Z_{12}$, so $V_1$ is a spectrum of $E$.

It remains to consider $V_2$. For $(x,y)\in E$, define
\[
  r=2x+y\pmod6,
  \qquad
  s=x+2y\pmod{12}.
\]
The two $s$-values in each $r$-fiber are
\begin{equation}\label{eq:r-s-table}
  \begin{array}{c|cccccc}
    r&0&1&2&3&4&5\\ \hline
    s&\{0,6\}&\{2,8\}&\{4,10\}&\{0,6\}&\{5,11\}&\{1,7\}.
  \end{array}
\end{equation}
Thus each $r\in\Z_6$ occurs twice, and the two corresponding $s$-values differ by $6$.

Write the elements of $V_2$ as
\[
  \lambda_{k,\delta}=k(4,2)+\delta(1,2),
  \qquad k\in\Z_6,
  \quad \delta\in\Z_2.
\]
Then
\[
  \lambda_{k,\delta}\cdot(x,y)=2kr+\delta s\pmod{12}.
\]
For each $r$, choose a representative $s_r$ from the corresponding pair in \eqref{eq:r-s-table}, and write the two values in that fiber as $s=s_r+6\varepsilon$, where $\varepsilon\in\Z_2$. With $\omega=\zeta^2=\ee^{2\pi i/6}$, the character matrix has entries
\[
  \omega^{kr}\zeta^{\delta s_r}(-1)^{\delta\varepsilon}.
\]
Consider two rows indexed by $(k,\delta)$ and $(k',\delta')$. If $\delta\ne\delta'$, summation over $\varepsilon\in\Z_2$ gives zero. If $\delta=\delta'$ but $k\ne k'$, summation over $\varepsilon$ gives a factor $2$, while
\[
  \sum_{r\in\Z_6}\omega^{(k-k')r}=0.
\]
Hence the twelve rows are pairwise orthogonal, so $V_2$ is a spectrum of $E$. The assertions with $E$ as the spectrum follow from \eqref{eq:spectral-symmetry}.
\end{proof}

\begin{proposition}\label{prop:no-common-complement}
There is no set $C\subseteq H$ such that
\[
  V_0+C=V_1+C=V_2+C=H.
\]
\end{proposition}

\begin{proof}
Assume that such a set $C$ exists. Necessarily $\card{C}=12$. Since $V_0$ is the vertical subgroup, $V_0+C=H$ implies that the first coordinates of the points in $C$ are all distinct. Hence
\[
  C=\{(x,g(x)):x\in\Z_{12}\}
\]
for some function $g:\Z_{12}\to\Z_{12}$. Since $V_1$ is the diagonal subgroup, $V_1+C=H$ implies that
\[
  h(x)=x-g(x)
\]
is a permutation of $\Z_{12}$.

Consider the homomorphism
\[
  \phi:H\longrightarrow\Z_4\times\Z_6,
  \qquad
  \phi(x,y)=(x\bmod4,\ x+y\bmod6).
\]
Its kernel is $H_0=\langle(4,2)\rangle$, and $\phi(c)=(1,3)$. Put $P=\phi(C)$. The restriction $\phi|_C$ is injective: if $\phi(c_1)=\phi(c_2)$, then
\[
  c_1-c_2\in H_0\subseteq V_2-V_2
  \quad\text{and}\quad
  c_1-c_2\in C-C,
\]
so the directness of $V_2+C=H$ and \eqref{eq:tile-difference} force $c_1=c_2$. Hence $\card P=\card C=12$. Applying $\phi$ to $V_2+C=H$ gives the set-theoretic covering
\[
  \phi(V_2)+P=\Z_4\times\Z_6.
\]
Here $\phi(V_2)=\{(0,0),(1,3)\}$, and
\[
  \card{\phi(V_2)}\card P=2\cdot12=24=\card{\Z_4\times\Z_6}.
\]
Therefore the addition map is bijective, so the quotient factorization is direct:
\begin{equation}\label{eq:P-two-tile}
  P+\{(0,0),(1,3)\}=\Z_4\times\Z_6.
\end{equation}

For $r\in\Z_4$, let
\[
  S_r=\{s\in\Z_6:(r,s)\in P\}.
\]
The first coordinates of $C$ run through $\Z_{12}$, so each $S_r$ has three elements. Equation \eqref{eq:P-two-tile} yields
\[
  S_r\ \dot\cup\ (S_{r-1}+3)=\Z_6.
\]
Let $e_r$ denote the number of even elements of $S_r$. Translation by $3$ reverses parity, and $\Z_6$ has three even elements. Therefore
\[
  e_r+(3-e_{r-1})=3,
  \qquad e_r=e_{r-1}.
\]
All four values $e_r$ are equal, so the number of points in $P$ whose second coordinate is even is divisible by $4$.

On the other hand, the second coordinate of $\phi(x,g(x))$ is $x+g(x)$ modulo $6$. The integers $x+g(x)$ and $x-g(x)=h(x)$ have the same parity. Since $h$ is a permutation of $\Z_{12}$, exactly six of these values are even. The same quantity is therefore both divisible by $4$ and equal to $6$, a contradiction.
\end{proof}

\subsection{Five-layer amplification and passage to the plane}

Define
\[
  \widetilde A_+
  =(V_0\times\{0,1,2\})
  \cup(V_1\times\{4\})
  \cup(V_2\times\{3\})
\]
and
\[
  \Gamma=E\times\Z_5\subseteq\wh{K}.
\]

\begin{proposition}\label{prop:Atilde-plus}
The set $\Gamma$ is a spectrum of $\widetilde A_+$, but $\widetilde A_+$ does not tile $K$.
\end{proposition}

\begin{proof}
The spectral assertion is the first conclusion of \Cref{lem:layer-common-complement}. Indeed, \Cref{lem:three-spectra} says that $E$ is a spectrum of each of $V_0,V_1,V_2$, and the layer assignment is
\[
  \jmath(0)=\jmath(1)=\jmath(2)=0,
  \qquad
  \jmath(3)=2,
  \qquad
  \jmath(4)=1.
\]
Thus $\Gamma=E\times\Z_5$ is a spectrum of $\widetilde A_+$.

We verify the difference-set condition needed for the second conclusion of that lemma. First,
\begin{equation}\label{eq:V0-minus-V1}
  V_0-V_1=H,
\end{equation}
because
\[
  (a,b)=(0,b-a)-(-a,-a)
\]
for every $(a,b)\in H$. The $V_0$-layers are $\{0,1,2\}$ and the $V_1$-layer is $\{4\}$, while
\[
  (\{0,1,2\}-4)\cup(4-\{0,1,2\})
  =\Z_5\setminus\{0\}.
\]
Combining this identity with \eqref{eq:V0-minus-V1} gives
\begin{equation}\label{eq:full-nonzero-slabs}
  H\times(\Z_5\setminus\{0\})
  \subseteq\widetilde A_+-\widetilde A_+.
\end{equation}
By \Cref{prop:no-common-complement}, the sets $V_0,V_1,V_2$ have no common tiling complement. The second conclusion of \Cref{lem:layer-common-complement}, applied with \eqref{eq:full-nonzero-slabs}, now shows that $\widetilde A_+$ does not tile $K$.
\end{proof}

Apply the Chinese remainder isomorphism \eqref{eq:crt-physical} to $\widetilde A_+$ and define
\[
  A_+=\Phi(\widetilde A_+).
\]
Explicitly,
\begin{align*}
  A_+={}&\{(36t,k):t\in\{0,1,2\},\ k\in\Z_{12}\}\\
  &\cup\{(25k+24,k):k\in\Z_{12}\}\\
  &\cup\{(25u+48,v):(u,v)\in V_2\}.
\end{align*}
The dual isomorphism associated with \eqref{eq:crt-physical} is
\begin{equation}\label{eq:crt-dual}
  (u,v,r)\longmapsto(5u+12r\bmod60,\ v\bmod12).
\end{equation}
Indeed, if $X=25x+36t$, then
\[
  \frac{(5u+12r)X}{60}
  \equiv\frac{ux}{12}+\frac{rt}{5}\pmod1.
\]
Consequently, the image of $\Gamma=E\times\Z_5$ is
\begin{equation}\label{eq:Lambda-plus-explicit}
  \Lambda_+
  =\{(5f(y)+12r,\ y):y\in\Z_{12},\ r\in\Z_5\}
  \subseteq\wh{\mathcal G}.
\end{equation}

\begin{proof}[Proof of \Cref{thm:main-spectral}]
By \Cref{prop:Atilde-plus} and the Chinese remainder isomorphisms \eqref{eq:crt-physical} and \eqref{eq:crt-dual}, the set $\Lambda_+$ in \eqref{eq:Lambda-plus-explicit} is a spectrum of $A_+$, whereas $A_+$ does not tile $\mathcal G$.
\end{proof}

\begin{corollary}\label{cor:plane-spectral}
The set
\[
  \Omega_+^{(97)}=A_+[97]+[0,1)^2
\]
is a bounded spectral set in $\R^2$ that does not tile by translations. It is the disjoint union of $564{,}540$ half-open unit squares.
\end{corollary}

\begin{proof}
For $A_+\subset\Z_{60}\times\Z_{12}$, the quantitative bound in \Cref{prop:transference} is
\[
  \frac{8\cdot60\cdot12}{\card{A_+}}=
  \frac{8\cdot60\cdot12}{60}=96.
\]
Since $A_+$ does not tile the finite group, $A_+[97]$ does not tile $\Z^2$. On the other hand, $A_+$ is spectral, and the forward construction \eqref{eq:explicit-spectrum-lift} is valid for every $k\ge1$. Hence $A_+[97]$ is spectral, with spectrum
\begin{equation}\label{eq:Lambda-plus-lattice}
  \left\{
  \left(\frac{U+s/97}{60},
        \frac{W+t/97}{12}\right):
  (U,W)\in\Lambda_+,\ 0\le s,t<97
  \right\}.
\end{equation}
Part~\textnormal{(iii)} of \Cref{prop:transference} now shows that $\Omega_+^{(97)}$ is spectral and does not tile $\R^2$; a Euclidean spectrum is obtained by adding $\Z^2$ to the set in \eqref{eq:Lambda-plus-lattice}. Finally,
\[
  \card{A_+[97]}=\card{A_+}\,97^2=60\cdot97^2=564{,}540,
\]
which is the number of unit squares.
\end{proof}

\begin{proof}[Proof of \Cref{main-thm}]
The two assertions follow from \Cref{cor:plane-tile,cor:plane-spectral}.
\end{proof}

\section*{Acknowledgements}
Tao Zhang is partially supported by the National Natural Science Foundation of China (Grant No.~12571357) and the Natural Science Basic Research Program of Shaanxi (Program No.~2025JC-YBMS-048).

\end{document}